\documentclass[a4paper]{amsart}
\usepackage[utf8]{inputenc}

\usepackage{amsmath,amssymb,mathtools}
\usepackage[left=2.5cm,right=2.5cm,top=2cm,bottom=2cm]{geometry}
\usepackage[
	pagebackref, 
	colorlinks=true,
	citecolor=red
]{hyperref}

\newcommand{\stoptocwriting}{%
  \addtocontents{toc}{\protect\setcounter{tocdepth}{-5}}}
\newcommand{\resumetocwriting}{%
  \addtocontents{toc}{\protect\setcounter{tocdepth}{\arabic{tocdepth}}}}
\usepackage{enumitem}
\newtheorem{theorem}{Theorem}[section]
\newtheorem{lemma}[theorem]{Lemma}
\newtheorem{corollary}[theorem]{Corollary}
\newtheorem{proposition}[theorem]{Proposition}

\theoremstyle{definition}
\newtheorem{definition}[theorem]{Definition}
\newtheorem{problem}[theorem]{Problem}

\theoremstyle{remark}
\newtheorem{remark}[theorem]{Remark}

\numberwithin{equation}{section}

\newcommand*{\RR}{\mathbb{R}}
\newcommand*{\NN}{\mathbb{N}}
\newcommand*{\ZZ}{\mathbb{Z}}

\DeclareMathOperator{\EE}{\mathbb{E}} 
\DeclareMathOperator{\PP}{\mathbb{P}} 
\DeclareMathOperator{\Ent}{Ent} 
\DeclareMathOperator{\Var}{Var} 
\DeclareMathOperator{\calH}{\mathcal{H}} 
\DeclareMathOperator{\calE}{\mathcal{E}} 

\title{P-log-Sobolev inequalities on $\NN$}
\author{Bart{\l}omiej Polaczyk}
\email{b.polaczyk@mimuw.edu.pl}
\thanks{Research partially supported by the National Science Centre, Poland, via the Preludium grant no.\ 2020/37/N/ST1/02667.}

\begin{document}
\begin{abstract}
	We answer an open problem posed by Mossel--Oleszkiewicz--Sen regarding relations between $p$-log-Sobolev inequalities for $p\in(0,1]$.
	We show that for any interval $I\subset(0,1]$, there exist $q,p\in I$, $q<p$, and a measure $\mu$ for which the $q$-log-Sobolev inequality holds, while the $p$-log-Sobolev inequality is violated.
	As a tool we develop certain necessary and closely related sufficient conditions characterizing those inequalities in the case of birth-death processes on $\NN$.
\end{abstract}
\maketitle
\tableofcontents
\section{Preliminaries}
\subsection{Background}
Functional inequalities are one of the central objects of modern probability theory.
They arise naturally when studying mixing times of Markov chains and are an important tool in proving concentration and hypercontractive estimates.
Arguably, the most prominent examples are the Poincar\'{e} and log-Sobolev inequalities since most of the other functional inequalities known in the literature arise either as a modification of one of them (e.g., various modified log-Sobolev inequalities or inequalities with defects) or as a product of some sort of procedure that interpolates between the above two.

An important example of the latter type is a family of Beckner inequalities introduced by Beckner~\cite{MR954373} who showed they hold true in the case of the Gaussian measure.
They were further studied, e.g., by Lata{\l}a and Oleszkiewicz~\cite{MR1796718} who used them to obtain intermediate concentration estimates between subexponential and subgaussian. 
Another particular example is a family of $p$-log-Sobolev inequalities introduced by Gross~\cite{MR420249} for $p>1$.
His definition was then extended by Bakry~\cite{MR1307413} to any real $p$.
Mossel et al.~\cite{MosselOleszkiewiczSen} studied $p$-log-Sobolev inequalities in the context of reverse hypercontraction.

A vibrant area of research is the study of the relations between various inequalities.
E.g., it is by now classical that the log-Sobolev inequality implies the Poincar\'{e} inequality.
Such sort of results can often have far-reaching consequences as demonstrated, e.g., in~\cite{MosselOleszkiewiczSen}, where the authors prove that for $0\le q<p\le 2$, the $p$-log-Sobolev inequality is stronger than the $q$-log-Sobolev inequality (cf. Theorem~\ref{T:MOS} below).
They exploit this relation further to prove the reverse hypercontraction for measures satisfying the modified log-Sobolev inequality, which found various applications, cf., e.g.,~\cite{DBLP:journals/tit/Raginsky16,MR3876877}.
A similar situation occurred in~\cite{MR4372142}, where the authors proved that the modified log-Sobolev inequality implies a particular family of Beckner inequalities, which allows deriving strong concentration and moment bounds from the modified log-Sobolev inequality.
Both of the mentioned results are examples of positive results (i.e., claiming that one inequality implies the other) and follow from abstract arguments that involve direct comparison of Dirichlet forms.
Similar results can be found in, e.g., the works of Diaconis--Saloff-Coste~\cite{MR1410112} or Bobkov--Tetali~\cite{MR2283379}.

Another approach needs to be taken in the case of negative results, i.e., when showing that one inequality does not imply another.
E.g., it is classical that the Poincar\'{e} inequality does not imply the log-Sobolev inequality. 
In that case, one counterexample is the exponential measure which satisfies the Poincar\'{e} inequality and does not satisfy the log-Sobolev inequality as the latter implies subgaussian concentration.
To prove such statements, it is often useful to derive some characterization of the functional inequality in question.
These characterizations are often of independent interest, as the conditions they are expressed in are usually much more accessible than a direct proof of the functional inequalities.
Some of the results that provide such characterizations include: the work by Bobkov and G\"otze~\cite{MR1682772} who, by viewing the log-Sobolev inequality in the general framework of Orlicz spaces, characterize the log-Sobolev inequality on $\RR$; the work by Miclo~\cite{miclo} who uses Hardy inequalities on $\ZZ$ to characterize the Poincar\'{e} inequality on trees
or the work by Barthe--Roberto~\cite{bartheroberto} who treat the case of the modified log-Sobolev inequalities on $\RR$.

The aim of this note is to answer an open question from~\cite{MosselOleszkiewiczSen} on the relation between $p$-log-Sobolev inequalities for $p\in(0,1]$, cf. Problem~\ref{Pr:MOS} below.
Our result falls into the category of negative examples described above.
As a byproduct, we develop a sufficient condition and complement it with a closely related necessary condition for $p$-log-Sobolev inequalities, which are of independent interest.
Below we describe our setting and results in more detail.

\subsection{General setup}\label{sec:general_setup}
Let $(\Omega, \mathcal{B}, \mu)$ denote some discrete probability space.
We assume that $\mu$ is fully supported on $\Omega$.
Let $P\colon [0,\infty)\times \Omega \times\mathcal{B} \to [0,1]$ be a homogeneous Markov transition function for which $\mu$ is an invariant measure.
We assume that $P$ is reversible with respect to $\mu$ and that it induces a strongly continuous semigroup  $(P_t)_{t\ge 0}$ of operators on $L_2(\Omega, \mu)$, defined as 
\[
	P_tf(x) 
	= 
	\int_\Omega 
	f(y)
	P(t, x, dy).	
\]
It can be then shown that for each $f\in L_2(\Omega,\mu)$, the mapping
\[
	(0,\infty)
	\ni
	t\mapsto 
	\frac{1}{2t}
	\int_\Omega\int_\Omega
	(f(x)-f(y))^2
	P(t,x,dy)\mu(dx)
\]
is non-increasing.
Denote 
\begin{equation}\label{eq:calh_def}
	\calH 
	= 
	\Bigl\{\,
		f\in L_2(\Omega,\mu)
		\,\colon\,
		\sup_{t\ge 0}
		\frac{1}{2t}
		\int_\Omega\int_\Omega
		(f(x)-f(y))^2
		P(t,x,dy)\mu(dx)
		< 
		\infty
	\,\Bigr\}	
\end{equation}
to be a domain of the Dirichlet form $\calE$ associated with this semigroup, given by the formula 
\begin{equation}\label{eq:cale_def}
	\calE(f,g)
	=
	\lim_{t\to 0^+}
	\frac{1}{2t}
	\int_\Omega\int_\Omega
	(f(x)-f(y))
	(g(x)-g(y))
	P(t,x,dy)\mu(dx)
\end{equation}
for $f,g\in\calH$.
If $L$ is the infinitesimal generator of the semigroup $(P_t)_{t\ge 0}$, defined via 
\[
	Lf 
	= 
	\lim_{t\to 0^+}
	\frac{P_t f - f}{h}	
\]
with the convergence in the $L_2$ sense, then for $f,g$ belonging to the domain $\calH_L$ of $L$, 
\[ 
	\calE(f,g) = - \int fLg\,d\mu.
\]
We refer the Reader to~\cite{MR2091955,MR2778606,MR2574430,MR3155209,MR1258492,MR4372142} for a detailed treatment of Markov processes, generators, Dirichlet forms and associated domains.

Set $\calH_+= \calH \cap \RR_+^{\Omega}$, where $\RR_+ = (0,\infty)$.
Below we gather some notation and results from~\cite{MosselOleszkiewiczSen}.
Note that while the results in~\cite{MosselOleszkiewiczSen} are stated for finite spaces, the extension to general discrete spaces is straightforward as discussed in Section~12.3 therein.
\begin{definition}\label{def:ineqs}
For $p\in\RR\setminus\{0,1\}$, the $p$-log-Sobolev inequality is satisfied with constant $C>0$ if
\begin{equation}\label{eq:p-LS}
	\Ent_\mu(f^p) \le \frac{Cp^2}{4(p-1)}\calE(f^{p-1},f),
\end{equation}
for all $f\in\calH_+$, such that $f^{p-1}\in\calH_+$.

The 1-log-Sobolev inequality is satisfied with constant $C>0$ if
\begin{equation}\label{eq:1-LS}
	\Ent_\mu(f) \le \frac{C}{4}\calE(f,\log f)
\end{equation}
for all $f\in\calH_+$, such that $\log f\in\calH$.

The 0-log-Sobolev inequality is satisfied with constant $C>0$ if
\begin{equation}\label{eq:0-LS}
	\Var_\mu(\log f) \le \frac{C}{2}\calE(f,-1/f)
\end{equation}
for all $f\in\calH_+$, such that $1/f\in\calH_+$.

The Poincar\'{e} inequality is satisfied with constant $C>0$ if 
\begin{equation}\label{eq:Poincare}
	\Var_\mu(f) \le \frac{C}{2}\calE(f,f)
\end{equation}
for all $f\in\calH$.
\end{definition}

We write $p$-LS($C$) for short to denote the $p$-log-Sobolev inequality with constant $C$.
We say that a pair $(\mu, \calE)$ satisfies the $p$-log-Sobolev inequality if $p$-LS($C$) is satisfied with some finite $C>0$.
If the underlying Dirichlet form $\calE$ is clear from context, we omit it and simply say the $\mu$ satisfies the $p$-log-Sobolev inequality.
Note that the $1$-log-Sobolev inequality is often referred to as the \emph{modified} log-Sobolev (or entropic) inequality in the literature and that the $0$-log-Sobolev and $1$-log-Sobolev inequalities are limiting cases for $p$-log-Sobolev inequalities for $p\in\RR\setminus\{0,1\}$.

Results below provide some basic relations between the inequalities introduced in Definition~\ref{def:ineqs}.

\begin{proposition}[{\cite[Lemma 3.1]{MosselOleszkiewiczSen}}]\label{P:MOS-Poincare-equivalence}
	The Poincar\'{e} inequality~\eqref{eq:Poincare} with constant $C>0$ is equivalent to the 0-log-Sobolev inequality with the same constant~$C$.
\end{proposition}

We set $p'=p/(p-1)$ for the H\"{o}lder conjugate of $p\in\RR\setminus\{1\}$.
\begin{proposition}[{\cite[Lemma 3.2]{MosselOleszkiewiczSen}}]\label{P:holder-equivalence}
	For $p\in\RR\setminus\{1\}$, $p$-LS(C) is equivalent to $p'$-LS(C).
	\end{proposition}

In particular, Proposition~\ref{P:holder-equivalence} implies that the study of relations between various $p$-log-Sobolev inequalities can be reduced to the case $p\in[0,2]$.
Moreover, since in the continuous setting (in the presence of the chain rule) all $p$-log-Sobolev inequalities for $p\neq 0$ are equivalent to the usual log-Sobolev inequality\footnote{This can be seen by substituting $f^2 \leftarrow f^p$ in the definition of the $p$-log-Sobolev inequality for $p\in\RR\setminus\{0\}$.} and since by Proposition~\ref{P:MOS-Poincare-equivalence}, 0-log-Sobolev inequality is equivalent to the Poincar\'{e} inequality which is strictly weaker than the log-Sobolev inequality, the only interesting range of $p$ is in fact $p\in(0,2]$.

Denote
\begin{equation}\label{eq:calep_def}
	\calE_p(f)
	=
	\begin{cases}
		pp'\calE(f^{1/p}, f^{1/p'})
		& \text{if } p\in  (0,2] \setminus \{1\},\\
		\calE(f,\log f)
		& \text{if } p =1,
	\end{cases}
\end{equation}
so that $p$-LS($C$) for $p\in (0,2]$ is equivalent to 
\[ 
	\Ent_\mu(f)\le \frac{C}{4}\calE_p(f).
\]
\begin{proposition}[{\cite[Theorem 2.1]{MosselOleszkiewiczSen}}]\label{P:MOS_main}
For any positive $f$, the mapping $(0,2] \ni p\mapsto \calE_p(f)$ is non-increasing.
\end{proposition}

\begin{remark}
	Note that~\cite[Theorem 2.1]{MosselOleszkiewiczSen} is stated for $p\in(0,2]\setminus\{1\}$ but a natural extension to the case $p=1$ is straightforward as discussed in Remark~2.2 therein.
\end{remark}

Proposition~\ref{P:MOS_main} serves as a tool for obtaining the following main result of~\cite{MosselOleszkiewiczSen}.
\begin{theorem}[{\cite[Theorem 1.7]{MosselOleszkiewiczSen}}]\label{T:MOS}
For any $0\le q\le p\le 2$, $p$-LS($C$) implies $q$-LS($C$).
Moreover, for any $1<q\le p \le 2$, $q$-LS($C$) implies $p$-LS($\tilde{C}$) with $\tilde{C}=Cqq'/pp'$.
\end{theorem}

The main goal of this paper is to provide an answer to the following question posed in~\cite[Section 12]{MosselOleszkiewiczSen}.
\begin{problem}\label{Pr:MOS}
	Is there any subset $I\subset (0,1]$ with non-empty interior, such that for  any $p,q\in I$, $p$-LS$(C)$ implies $q$-LS$(c(I)C)$, where $c(I)>0$ depends on $I$ only?
\end{problem}

\section{Main result}\label{sec:main-result}
Let $\mu$ be a measure on $\NN$ with full support.
We use the convention that $\mu_k=\mu(\{k\})$ and $\mu[k,\infty) = \mu([k,\infty))$.
For any $f\colon\NN\to\RR$ and $k\in\NN$, denote $Df(k) = f(k+1) - f(k)$. 
Consider the following Dirichlet form
\begin{align}\label{eq:calE_birth-death}
	\calE(f,g) 
	=
	\sum_{k\ge 0} Df(k)Dg(k)\mu_k
\end{align}
defined for 
$
	f,g\in
	\{\,
		h\in L_2(\Omega,\mu)\colon \sum_{k\ge 0} (Dh(k))^2\mu_k < \infty
	\,\}
	=:
	\calH
	$.
The corresponding birth-death dynamics generator is given by the formula
\begin{equation}\label{eq:logSobolev-generator}
	Lf(k) 
	= 
	Df(k)
	- 
	{\bf 1}_{\{ k>0 \} }
	\frac{\mu_{k-1}}{\mu_k}
	Df(k-1),
\end{equation}
so that $\calE(f,g)=-\int fLg\,d\mu$ for $f,g\in\calH_L\subset\calH$.
A Markov process generated by $L$ given by~\eqref{eq:logSobolev-generator} and Dirichlet form given by~\eqref{eq:calE_birth-death} with stationary measure $\mu$ being a geometric measure is investigated in~\cite{MR1944012} as an example of a pair $(\mu,\calE)$ satisfying the Poincar\'{e} inequality~\eqref{eq:Poincare} and violating the modified log-Sobolev inequality~\eqref{eq:1-LS}.

\begin{remark}
	It follows from the general theory of birth and death processes (see, e.g.,~\cite{MR2091955,MR2574430}), that $\mathcal{E}$ is indeed a Dirichlet form corresponding to a Markov process on $\NN$.
	As a consequence, when proving a functional inequality in question it actually suffices to consider simple (i.e., having finitely many jumps) functions.
	Therefore, henceforth we will sometimes restrict from specifying particular domains and simply assume that the expressions we introduce are considered for functions for which they are well-defined as one can always restrict the attention to the class of simple functions.
\end{remark}

For $x>0$, set 
\begin{equation}\label{eq:H_p_def}
	H_p(x) = 
	\begin{cases}
		pp'(x^{1/p}-1)(x^{1/p'}-1)
		& \text{if } p\in(0,1),\\
		(x-1)\log(x)
		& \text{if } p=1
	\end{cases}	
\end{equation}
so that $\calE_p$ defined in~\eqref{eq:calep_def} is given by 
\begin{equation}\label{eq:calep_birth_death}
	\calE_p(f) 
	=
	\sum_{k\ge 0} f(k) H_p\bigl(\frac{f(k+1)}{f(k)}\bigr)\mu_k
\end{equation}
and $p$-LS($C$) for $p\in (0,1]$ is equivalent to
\begin{equation}\label{eq:pLS_birth_death}
	\Ent_\mu(f) \le  
	\frac{C}{4}
	\sum_{k\ge 0} f(k) H_p\bigl(\frac{f(k+1)}{f(k)}\bigr)\mu_k,
\end{equation}
while $0$-LS($C$) (i.e., the Poincar\'{e} inequality) is equivalent to 
\begin{equation}\label{eq:poincare_N}
	\Var_\mu(f) 
	\le \frac{C}{2}\sum_{l=0}^\infty (Df(l))^2\mu_l.
\end{equation}

Theorem below is our main result.
\begin{theorem}\label{C:main}
	For any $p\in(0,1)$, there exists a measure $\mu$ on $\NN$ that does not satisfy the $p$-log-Sobolev inequality but satisfies the $q$-log-Sobolev inequality for all $q\in (0, p)$.
	In particular, there are no intervals~$I$ that meet the conditions posed in Problem~\ref{Pr:MOS}.
\end{theorem}
We construct the required counterexample and verify that it satisfies the appropriate $p$-log-Sobolev inequalities with the use of the theorem below, which is of independent interest.
\begin{theorem}\label{T:main}
	Choose any $p\in(0,1]$ and a measure $\mu$ on $\NN$ with full support and with associated Dirichlet form $\calE$, given by the birth-death process generator~\eqref{eq:calE_birth-death}. 
	
	If $\mu$ satisfies the Poincar\'{e} inequality~\eqref{eq:poincare_N} with some finite constant $C_P>0$ and 
	\begin{equation}\label{eq:main_condition}
		\hat{C}
		:= 
		\sup_{n\ge 1}
		\Big\{
			\Bigl[
			H_p\bigl( 
				\frac{\mu[n-1,\infty)}{\mu[n,\infty)}
			\bigr)
			\Bigr]^{-1}
			\cdot 
			\log\bigl(
				\frac{2}{\mu[n,\infty)}
			\bigr)
		\Big\}
		< \infty,
	\end{equation}
	then $\mu$ satisfies $p$-LS($C$)~\eqref{eq:pLS_birth_death} with some finite $C>0$.

	Contrarily, if there exists an increasing sequence $\tau_0<\tau_1<\ldots$ such that 
	\begin{equation}\label{eq:necessary_cond}
		\lim_{n\to\infty}
		\Big\{
			\Bigl[
			H_p\bigl( 
				\frac{\mu[\tau_{n-1},\infty)}{\mu[\tau_n,\infty)}
			\bigr)
			\Bigr]^{-1}
			\cdot 
			\frac{\mu[\tau_{n-1},\infty)}{\mu[\tau_{n}-1,\infty)}
			\cdot 
			\log\bigl(
				\frac{2}{\mu[\tau_n,\infty)}
			\bigr)
		\Big\}
		=
		\infty,
	\end{equation}
	then $\mu$ does not satisfy $p$-LS($C$)~\eqref{eq:pLS_birth_death} with any finite $C>0$.
\end{theorem}
\begin{remark}
	The Poincar\'{e} inequality is implied by the $p$-log-Sobolev inequality for any $p\in (0,2]$, cf. Theorem~\ref{T:MOS} and Proposition~\ref{P:MOS-Poincare-equivalence}, therefore making it a part of the sufficient condition of Theorem~\ref{T:main} (alongside~\eqref{eq:main_condition}) is non-restrictive.
\end{remark}
\begin{remark}
	The negation of~\eqref{eq:main_condition} is equivalent to the existence of an increasing sequence $\tau_0<\tau_1<\ldots$ such that 
	\[
		\lim_{n\to\infty}
		\Big\{
			\Big[
			H_p\bigl( 
				\frac{\mu[\tau_{n}-1,\infty)}{\mu[\tau_n,\infty)}
			\bigr)
			\Big]^{-1}
		\cdot 
		\log\bigl(
			\frac{2}{\mu[\tau_n,\infty)}
		\bigr)
		\Big\}
		=
		\infty.
	\] 
	If $\mu$ verifies the Poincar\'{e} inequality~\eqref{eq:poincare_N} with constant $C_P$, then by Lemma~\ref{L:rho_lower_bound} (cf. also Remark~\ref{rk:hardy-poincare}) and by Lemma~\ref{L:H_p_properties},~\ref{L:H_p_c_increasing} below, for any $\tau_{n-1}<\tau_n$ and some $c>0$,
\[
	\Bigl[
		H_p\bigl( 
			\frac{\mu[\tau_{n}-1,\infty)}{\mu[\tau_n,\infty)}
		\bigr)
		\Bigr]^{-1}
	\ge 
	\frac{1}{1+c}
	\Bigl[
	H_p\bigl( 
		\frac{\mu[\tau_{n-1},\infty)}{\mu[\tau_n,\infty)}
	\bigr)
	\Bigr]^{-1}
	\cdot 
	\frac{\mu[\tau_{n-1},\infty)}{\mu[\tau_{n}-1,\infty)}.
\]
Whence, condition~\eqref{eq:necessary_cond} implies that $\hat{C}=\infty$ and thus~\eqref{eq:main_condition} does not hold, but the reverse needs not to be true.
Thus, providing a full characterization of $p$-log-Sobolev inequalities on $\NN$ for $p\in(0,1]$ remains open.
\end{remark}
\section{Auxiliary results}
In this section we gather some lemmas and known results required for the proof of Theorem~\ref{T:main}.

\subsection{Hardy inequality}
In the sequel we put $0\cdot\infty = 0$ and $c/0=\infty$ for any $c>0$.
\begin{definition}
We say that a probability measure $\mu$ on $\NN$ satisfies the Hardy inequality with constant $C$, if 
\begin{equation}\label{eq:Hardy}
	\sum_{l=0}^\infty (f(l)-f(0))^2\mu_l
	\le 
	C 
	\sum_{l=0}^\infty (Df(l))^2\mu_l
	=
	C
	\calE(f,f)
\end{equation}
for all $f\in\calH$.
\end{definition}

Define 
\begin{equation}\label{eq:C_mu}
	C_\mu
	= 
	\sup_{k\ge 1}
	\Bigl\{
		\mu[k,\infty)
		\cdot
		\sum_{l=0}^{k-1} \frac{1}{\mu_l}
	\Bigr\}.
\end{equation}
The following result states that $C_\mu$ characterizes the Hardy inequality~\eqref{eq:Hardy}.

\begin{theorem}[Miclo~\cite{miclo}]\label{T:Miclo}
	The best constant $\hat{C}_H$ in the Hardy inequality~\eqref{eq:Hardy},
	\begin{equation}\label{eq:C_H_best}
		\hat{C}_H =
		\sup_{f\in\calH}
		\Bigl\{
			\frac{\sum_{l=0}^\infty (f(l)-f(0))^2\mu_l}{\calE(f,f)}
			\colon \calE(f,f)>0
		\Bigr\},
	\end{equation}
	verifies
	\begin{equation*}
		C_\mu 
		\le 
		\hat{C}_H 
		\le 
		4C_\mu.
	\end{equation*}
\end{theorem}
\begin{remark}\label{rk:hardy-poincare}
	It is easy to see that for fully supported measures, the Poincar\'{e} inequality~\eqref{eq:poincare_N} is satisfied if and only if the Hardy inequality~\eqref{eq:Hardy} is satisfied.
	More precisely, the best constant $\hat{C}_P$ in the Poincar\'{e} inequality~\eqref{eq:poincare_N}, 
	\begin{equation}\label{eq:C_P_best}
		\hat{C}_P = 
		2
		\sup_{f\in\calH}
		\Bigl\{
			\frac{\Var_\mu(f)}{\calE(f,f)}
			\colon
			\calE(f,f)>0
		\Bigr\},
	\end{equation}
	satisfies 
	\[
		2\mu_0 \hat{C}_H
		\le 
		\hat{C}_P
		\le 
		2\hat{C}_H.
	\]
	Indeed, $\hat{C}_P \le 2\hat{C}_H$ follows from the estimate $\Var_\mu(f) \le \int (f-f(0))^2\,d\mu$.
	To see that $2\mu_0 \hat{C}_H \le \hat{C}_P$, let $f_\varepsilon\in\calH$ for any $\varepsilon > 0$ be such that $f_\varepsilon(0)=0$, $\calE(f_\varepsilon,f_\varepsilon)>0$ and $\int f_\varepsilon^2\,d\mu \ge (\hat{C}_H-\varepsilon)\calE(f_\varepsilon, f_\varepsilon)$.
	Then, by the Cauchy--Schwarz inequality, $(\int f_\varepsilon\,d\mu)^2 \le (1-\mu_0) \int f_\varepsilon^2\,d\mu$, whence 
	\begin{align*}
		\frac{\hat{C}_P}{2}\calE(f_\varepsilon,f_\varepsilon)
		\ge
		\Var_\mu(f_\varepsilon)
		\ge 
		\mu_0 
		\int f_\varepsilon^2\,d\mu
		\ge 
		\mu_0
		(\hat{C}_H-\varepsilon)\calE(f_\varepsilon, f_\varepsilon),
	\end{align*}
	and we conclude by taking $\varepsilon\to 0^+$.
\end{remark}
The quantity $C_\mu$ is useful for controlling the tail behavior of $\mu$ as demonstrated in the lemmas below.
\begin{lemma}\label{L:rho_lower_bound}
	If $\mu$ is fully supported, then 
	\begin{equation}\label{eq:quotient_upper_bound}
		\sup_{k\ge 0} 
		\frac{\mu[k,\infty)}{\mu_k} 
		\le 
		1+C_\mu
	\end{equation}
	and
	\begin{equation}\label{eq:rho_lower_bound}
		\inf_{k\ge 1} 
		\Bigl\{ 
			\frac{\mu[k-1,\infty)}{\mu[k,\infty)} 
		\Bigr\}
		\ge 
		1+\frac{1}{C_\mu}.
	\end{equation}
\end{lemma}
\begin{proof}
	The estimate~\eqref{eq:quotient_upper_bound} follows from the definition~\eqref{eq:C_mu} of $C_\mu$.
	Using~\eqref{eq:quotient_upper_bound}, we obtain for any $k\ge 1$,
	\begin{align*}
		\mu[k,\infty) 
		&=
		\mu[k-1,\infty) - \mu_{k-1}
		\\&\le 
		\mu[k-1,\infty) 
		- 
		\frac{\mu[k-1,\infty)}{1+C_\mu}
		= 
		\bigl(
			1+\frac{1}{C_\mu}
		\bigr)^{-1}
		\mu[k-1,\infty)
	\end{align*}
	and~\eqref{eq:rho_lower_bound} follows.
\end{proof}

\begin{lemma}\label{L:poincare_implies_integrability}
	If $\mu$ is fully supported, then for $X\sim\mu$, $\EE X \le 1 + C_\mu$.
\end{lemma}
\begin{proof}
	By Lemma~\ref{L:rho_lower_bound},
	$
		\EE X 
		\le
		\sum_{k\in\NN} \mu[k,\infty)
		\le 
		(1+C_\mu)
		\sum_{k\in\NN} \mu_k
		= 
		1+C_\mu
	$.
\end{proof}

\subsection{Tail estimates}
Proposition below states that the $p$-log-Sobolev inequality implies Poisson-type tail behavior for $p\in(0,1]$.
It is deduced by using a variant of Herbst's argument.
\begin{proposition}\label{P:poisson_tails}
If $\mu$ is fully supported and satisfies $p$-LS($C$)~\eqref{eq:pLS_birth_death} for some $p\in(0,1]$ and some $C<\infty$, then there exists $\varepsilon_p\colon (0,\infty) \to [0,1]$, such that $\varepsilon_p(t)\to 0$ as $t\to\infty$ and
\begin{equation}\label{eq:poisson_tails}
	\log\bigl(
		\mu[t,\infty)
	\bigr)
	\le 
	-(1-\varepsilon_p(t)) tp\log (t+1)
\end{equation}
for any $t\ge 0$.
\end{proposition}
For the proof of Proposition~\ref{P:poisson_tails}, we need the following lemma.
\begin{lemma}\label{L:integral_est}
	For any $\lambda > 1$ and $p\in(0,1)$,
	\begin{equation}\label{eq:integral_est}
		\int_1^\lambda \frac{e^{s/p}-1}{s^2}\,ds 
		\le 
		\frac{p}{1-p}
		\cdot 
		\frac{e^{\lambda/p}}{\lambda}
		=
		-p'
		\cdot 
		\frac{e^{\lambda/p}}{\lambda}.
	\end{equation}
\end{lemma}
\begin{proof}
	For any $s>1$,
	\begin{equation*}
		\frac{1-p}{p}\cdot 
		\frac{e^{s/p}-1}{s^2}
		\le 
		\frac{1-p}{p}
		\cdot 
		\frac{e^{s / p}}{s^2}
		\le 
		\frac{s-p}{p}
		\cdot 
		\frac{e^{s / p}}{s^2}
		=
		\frac{d}{ds}\Bigl[
			\frac{e^{s/p}}{s}
		\Bigr],
	\end{equation*}
	whence
	\[
		\int_1^\lambda 
		\frac{e^{s/p}-1}{s^2}	
		\,ds 
		\le 
		\frac{p}{1-p}
		\Bigl(
			\frac{e^{\lambda/p}}{\lambda}
			- 
			e^{1/p}
		\Bigr)
		\le 
		\frac{p}{1-p}
		\frac{e^{\lambda/p}}{\lambda}
	\]
	as desired.
\end{proof}
\begin{proof}[Proof of Proposition~\ref{P:poisson_tails}]
Consider first the case $p\in(0,1)$.
For $s, h>0$, $x\in\RR$, denote $\phi_h(x) = \min(x,h)$, $f_{s,h}(x)=\exp(s \phi_h(x)/p)$ and $g_{s, h}(x) = \exp(s \phi_h(x)/p')$.
Let $X\sim\mu$ and for $h>0$, set $X_h = \min(X,h)$.
For any  $s,h>0$, $h\in\NN$, applying $p$-LS($C$) to $f_{s,h}$	 and using the Dirichlet form formula~\eqref{eq:calE_birth-death}, we obtain
\begin{align}\label{eq:X_exp_integrability_derivative}
\begin{split}
	\frac{d}{ds} \Bigl[
		\frac{\log \EE e^{s X_h}}{s}
	\Bigr]
	&= 
	\frac{\Ent (e^{s X_h})}{s^2 \EE e^{s X_h}}
	\\&\le 
	\frac{Cpp'}{4s^2 \EE e^{s X_h}}
	\calE(f_{s, h}, g_{s, h})
	\\&=
	\frac{Cpp'}{4s^2 \EE e^{s X_h}}
	\sum_{k=0}^{h-1} 
	\bigl(
		e^{s(k+1)/p}
		-
		e^{sk/p}
	\bigr)
	\bigl(
		e^{s(k+1)/p'}
		-
		e^{sk/p'}
	\bigr)\mu_k
	\\&=
	\frac{Cpp'}{4s^2}
	\bigl(
		e^{s/p} - 1
	\bigr)
	\bigl(
		e^{s/p'} - 1
	\bigr)
	\frac{\sum_{k=0}^{h-1}
	e^{sk/p + sk/p'}\mu_k}{\EE e^{s X_h}}
	\\&\le 
	\frac{Cpp'}{4s^2}
	\bigl(
		e^{s/p} - 1
	\bigr)
	\bigl(
		e^{s/p'} - 1
	\bigr),
\end{split}
\end{align}
where in the last inequality we have also used that $p'(e^{s/p'} - 1) > 0$.
By Theorem~\ref{T:MOS}, $\mu$ satisfies the Poincar\'{e} inequality~\eqref{eq:poincare_N}, whence by Lemma~\ref{L:poincare_implies_integrability}, $\EE X < 1+C_\mu < \infty$, cf. Remark~\ref{rk:hardy-poincare}.
Consequently, for any $\lambda,h>0$,
\begin{align}
\begin{split}\label{eq:X_exp_integrability}
	\frac{\log \EE e^{\lambda X_h}}{\lambda}
	&=
	\EE X_h 
	+
	\int_0^\lambda 
	\frac{d}{ds}\bigl[
		\frac{\log \EE e^{sX_h}}{s}
	\bigr]\,ds 
	\\&\le 
	\EE X_h 
	+
	\frac{Cpp'}{4}
	\int_0^\lambda 
	\frac{
		\bigl(
			e^{s/p} - 1
		\bigr)
		\bigl(
			e^{s/p'} - 1
		\bigr)
	}{s^2}
	\,ds 
	\\&\le 
	\EE X
	+
	\frac{Cpp'}{4}
	\int_0^\lambda 
	\frac{
		\bigl(
			e^{s/p} - 1
		\bigr)
		\bigl(
			e^{s/p'} - 1
		\bigr)
	}{s^2}
	\,ds
	< \infty.
\end{split}
\end{align}
Taking $h\to +\infty$ and using monotone convergence theorem in~\eqref{eq:X_exp_integrability}, we obtain that for any $\lambda > 0$, $\EE e^{\lambda X}<\infty$.
Therefore, we can repeat reasoning from~\eqref{eq:X_exp_integrability_derivative} to get that for any $s>0$,
\begin{equation}\label{eq:herbst}
	\frac{d}{ds} \Bigl[
		\frac{\log \EE e^{s X}}{s}
	\Bigr]
	\le 
	\frac{Cpp'}{4s^2}
	\bigl(
		e^{s/p} - 1
	\bigr)
	\bigl(
		e^{s/p'} - 1
	\bigr).
\end{equation}
For any $\lambda>1$, by~\eqref{eq:herbst}, estimating $1-e^{s/p'}\le 1$ for $s>1$ and using Lemma~\ref{L:integral_est},
\begin{align*}
	\int_1^\lambda
	\frac{d}{ds}
	\bigl[
		\frac{\log \EE e^{s X}}{s}
	\bigr]\,ds
	&\le 
	\frac{C pp'}{4}
	\int_1^\lambda
	\frac{(e^{s/p}-1)(e^{s/p'}-1)}{s^2}
	\,ds 
	\\&\le 
	\frac{C p(-p')}{4}
	\int_1^\lambda
	\frac{e^{s/p}-1}{s^2}
	\,ds 
	\le
	\frac{C pp'^2}{4\lambda}
	e^{\lambda/p},
\end{align*}
whence, for $t>e^{1/p}-1$, by the Chernoff bound and setting $\lambda = p\log (t+1)>1$,
\begin{align*}
	\log\bigl(
		\PP(X \ge t)		
	\bigr)
	&\le 
	\log \EE 
	\exp(-\lambda t + \lambda X)
	\\
	&=
		-\lambda t 
		+ 
		\lambda \log \EE e^X 
		+
		\lambda
		\int_1^\lambda
		\frac{d}{ds}
		\bigl[
			\frac{\log \EE e^{s X}}{s}
		\bigr]\,ds
	\\
	&\le
		-\lambda t 
		+ 
		\lambda \log \EE e^X 
		+
		\frac{C pp'^2}{4}
		e^{\lambda/p}
	\\
	&=
	- pt\log (t+1)
	+ p (\log \EE e^X)\log (t+1) 
	+
	\frac{C pp'^2}{4}(t+1) 
	\\
	&=
	\notag
	-\Bigl(
		1 - \frac{\log\EE e^X}{t} - \frac{Cp'^2(t+1)}{4t\log(t+1)}
	\Bigr)tp \log (t+1).
\end{align*}
Therefore, for any $p\in(0,1)$ and 
\begin{equation}\label{eq:varepsilon_p_def}
	\varepsilon_p(t)
	=
	\begin{cases}
		1 &\text{if}\quad t\le e^{1/p}-1,\\
		\min\bigl(
			1,
			\frac{\log\EE e^X}{t} 
			+
			\frac{C}{4}
			\cdot
			\frac{1}{(1/p-1)^2}
			\cdot
			\frac{t+1}{t\log(t+1)}
		\bigr)
		&\text{if}\quad t> e^{1/p}-1,
	\end{cases}
\end{equation}
the estimate~\eqref{eq:poisson_tails} holds for any $t\ge 0$ as desired.

We turn to the case $p=1$.
Recall that by Theorem~\ref{T:MOS}, for any $p\in(0,1)$ the 1-log-Sobolev inequality implies the $p$-log-Sobolev inequality with the same constant.
Let $\varepsilon_p$ be given by~\eqref{eq:varepsilon_p_def} and let $\varphi\colon (0,\infty)\to (0,1)$ be such that $\varphi(t)\to 1$ and $\varepsilon_{\varphi(t)}(t)\to 0$ as $t\to\infty$ (one can take, e.g., $\varphi(t)=0.5$ for $t\le e^e$ and $\varphi(t) = (1+1/\log\log t)^{-1}$ for $t>e^e$).
For $t\in(0,\infty)$, set 
\[ 
	\varepsilon_1(t) 
	= 
	1 -
	\varphi(t)
	\bigl(
		1 
		- 
		\varepsilon_{\varphi(t)}(t)
	\bigr).
\]
Then $\varepsilon_1(t)\to 0$ as $t\to\infty$ and for any $t> 0$, we apply~\eqref{eq:poisson_tails} with $p=\varphi(t)$ to obtain the conclusion.
\end{proof}

\subsection{Reduction to increasing functions}
The proposition below is an adaptation of \cite[Proposition~3]{bartheroberto} to the discrete case.
The proof goes along similar lines -- we present it for completeness.
\begin{proposition}\label{P:Barthe-Roberto}
	Recall the definition~\eqref{eq:C_mu} of $C_\mu$ and assume $C_\mu<\infty$.
	For any $f\colon\NN\to\RR_+$ with a finite number of jumps and for any $\rho>1$,
	\[
		\Ent_\mu(f) 
		\le 
		16C_\mu(1+\sqrt{\rho})^2
		\calE(\sqrt{f}, \sqrt{f})
		+
		\sum_{k\colon g(k)\ge \rho \EE_\mu g}
		g(k)\log\bigl(\frac{g(k)}{\EE_\mu g}\bigr)\mu_k,
	\]
	where
	$
		g(n)
		=
		f(0)
		+
		\sum_{k=0}^{n-1}(Df(k))_+.
	$
\end{proposition}
\begin{proof}
We use the assumption that $f$ has a finite number of jumps to make sure that every quantity below is well-defined.
Note that $g\ge f$ since $(Df(k))_+\ge Df(k)$ for all $k\in\NN$.
For $x,t>0$, denote $\psi(t,x)=x\log\frac{x}{t} - (x-t)$, $\phi(x)=x\log x$ and set 
$\Theta = \{ n\in \NN \colon f(n) \ge \rho \EE g \}$.
Convexity of $\phi$ implies that $\phi(\EE_\mu f) \ge \phi(\EE_\mu g) + \phi'(\EE_\mu g)(\EE_\mu f - \EE_\mu g)$, so that 
\begin{align}\label{eq:BR_ent_split}
	\begin{split}
	\Ent(f) 
	&= 
	\EE_\mu \bigl[
		\phi(f) - \phi(\EE_\mu f)
	\bigr]
	\\
	&\le 
	\EE_\mu\bigl[
		\phi(f)-\phi(\EE_\mu g) + \phi'(\EE_\mu g)\cdot (\EE_\mu g - \EE_\mu f)
	\bigr]
	= 
	\EE_\mu \psi(\EE_\mu g, f).
	\end{split}
\end{align}
Set $U_\rho=(1+\sqrt{\rho})^2$.
Since for $x\in [0,\rho t]$, 
\begin{align*}
	\psi(t,x)
	\le 
	x\Bigl(\frac{x}{t}-1\Bigr)
	-
	(x-t)
	&= 
	t\Bigl( \frac{x}{t}-1 \Bigr)^2
	\\&=
	\Bigl( 1+\sqrt{\frac{x}{t}} \Bigr)^2
	(\sqrt{x}-\sqrt{t})^2
	\le 
	U_\rho (\sqrt{x}-\sqrt{t})^2,
\end{align*}
we obtain
\begin{align}\label{eq:entropy_bound_small_f}
\begin{split}
	\EE_\mu \psi(\EE_\mu g, f){\bf 1}_{\Theta^c}
	&\le 
	U_\rho 
	\EE_\mu (\sqrt{f}-\sqrt{\EE_\mu g})^2
	\\&\le 
	2
	U_\rho 
	\bigl(
		\EE_\mu (\sqrt{f}-\sqrt{g})^2
		+
		\EE_\mu (\sqrt{g}-\sqrt{\EE_\mu g})^2
	\bigr)
	\\&\le 
	2
	U_\rho 
	\bigl(
		\EE_\mu (\sqrt{f}-\sqrt{g})^2
		+
		2\Var_\mu(\sqrt{g})
	\bigr).
\end{split}
\end{align}
Let $\hat{C}_H$ and $\hat{C}_P$ be the best constants in the Hardy inequality~\eqref{eq:Hardy} and the Poincar\'{e} inequality~\eqref{eq:poincare_N} respectively, cf. the definitions~\eqref{eq:C_H_best} and~\eqref{eq:C_P_best}.
By Theorem~\ref{T:Miclo} and as indicated in Remark~\ref{rk:hardy-poincare}, $\hat{C}_P \le 2\hat{C}_H\le 8C_\mu$.
Recall that $\sqrt{f(0)}=\sqrt{g(0)}$, whence the RHS of~\eqref{eq:entropy_bound_small_f} can be estimated from above by 
\begin{equation*}	
	2U_\rho
	\bigl(
		4C_\mu \sum_{k=0}^\infty (D\sqrt{f}-D\sqrt{g})^2(k)\mu_k
		+
		8C_\mu \sum_{k=0}^\infty (D\sqrt{g})^2(k)\mu_k
	\bigr)
	\le 
	16C_\mu U_\rho
	\calE(\sqrt{f}, \sqrt{f}),
\end{equation*}
where we have used the fact that $4(D\sqrt{f}-D\sqrt{g})^2+8(D\sqrt{g})^2\le 8(D\sqrt{f})^2$, which follows from the pointwise estimate $0\le D\sqrt{g}\le (D\sqrt{f})_+$ and the fact that for a fixed $y$, the convex mapping $x\mapsto 4(y-x)^2 + 8x^2$ on a closed interval admits a maximum at an end of this interval.

Turning to the remaining part of the RHS of~\eqref{eq:BR_ent_split}, we get 
\begin{align*}
	\EE_\mu \psi(\EE g, f){\bf 1}_\Theta
	&= 
	\EE_\mu
	\bigl[f\log \frac{f}{\EE_\mu g} 
	-
	(f-\EE_\mu g)
	\bigr]{\bf 1}_\Theta
	\\ 
	&\le 
	\EE_\mu \bigl[g \log\frac{g}{\EE_\mu g} \bigr]{\bf 1}_\Theta
	\le
	\sum_{k\colon g(k)\ge \rho \EE_\mu g} 
	g(k)\log\frac{g(k)}{\EE_\mu g}\mu_k,
\end{align*}
where in both inequalities we have used the definition of $\Theta$ and the facts that $f\le g$ and $\rho>1$.
Combining all the above estimates yields the conclusion.
\end{proof}

When proving that some $p$-log-Sobolev inequality is satisfied, Proposition~\ref{P:Barthe-Roberto} allows us to restrict our attention to a special subclass of functions from $\calH_+$.
This idea is formalized in the corollary below.

\begin{corollary}\label{C:Barthe-Roberto}
	If $C_\mu<\infty$ and
	\begin{equation}\label{eq:Barthe-Roberto_cond}
		\sum_{k\colon g(k)\ge \rho \EE_\mu g}
		g(k)\log\bigl(\frac{g(k)}{\EE_\mu g}\bigr)\mu_k
		\le 
		C_\rho
		\calE_p(g)
	\end{equation}
	for some $\rho>1$, $p\in(0,2]$, some constant $C_\rho>0$ and any non-decreasing function $g\colon\NN\to\RR_+$ with a finite number of jumps,
	then $\mu$ satisfies the $p$-log-Sobolev inequality~\eqref{eq:pLS_birth_death} with constant 
	\[
		C=64C_\mu(1+\sqrt{\rho})^2+4C_\rho.
	\]
\end{corollary}
\begin{proof}
	By applying Proposition~\ref{P:Barthe-Roberto} together with condition~\eqref{eq:Barthe-Roberto_cond}, we get that for any $f\colon\NN\to\RR_+$ with a finite number of jumps
	\begin{equation}\label{eq:f-g-replacement}
	\Ent_\mu(f) 
	\le 
	16C_\mu(1+\sqrt{\rho})^2
	\calE(\sqrt{f}, \sqrt{f})
	+
	C_\rho \calE_p(g),
	\end{equation}
	where $g(0)=f(0)$ and $g(n) = f(0) + \sum_{k=0}^{n-1} (Df(k))_+$ for $n\ge 1$.

	Note that $f\le g$ by definition and that for any $x, \Delta > 0$, the mapping $x\mapsto x H_p(1+\Delta/x)$ is non-increasing, which follows from the convexity of $H_p$ on $[1,\infty)$, cf. Lemma~\ref{L:H_p_properties},~\ref{L:H_p_monotone} below, and the fact that $H_p(1)=0$.
	Using that $H_p\ge 0$ and monotonicity of $x\mapsto x H_p(1+\Delta/x)$, we obtain
	\begin{align}\label{eq:calef-caleg-comp}
	\begin{split}
		\calE_p(f)
		&= 
		\sum_{k=0}^\infty f(k) H_p\Bigl( 1 + \frac{Df(k)}{f(k)} \Bigr)
		\mu_k 
		\\&\ge 
		\sum_{k=0}^\infty f(k) H_p\Bigl( 1 + \frac{Df(k)}{f(k)} \Bigr)
		\mu_k 
		{\bf 1}_{\{ Df(k) > 0 \}}
		\\&\ge 
		\sum_{k=0}^\infty g(k) H_p\Bigl( 1 + \frac{Df(k)}{g(k)} \Bigr)
		\mu_k 
		{\bf 1}_{\{ Df(k) > 0 \}}
		\\&=
		\sum_{k=0}^\infty g(k) H_p\Bigl( 1 + \frac{Dg(k)}{g(k)} \Bigr)
		\mu_k 
		=
		\calE_p(g).
 	\end{split}
	\end{align}

	By Proposition~\ref{P:MOS_main}, $\calE(\sqrt{f}, \sqrt{f}) \le \calE_2(f) \le \calE_p(f)$ for any $p\in (0,2]$.
	Therefore, combining~\eqref{eq:f-g-replacement} and~\eqref{eq:calef-caleg-comp}, we obtain that for any $f\colon\NN\to\RR_+$ with a finite number of jumps,
	\begin{equation}\label{eq:Barthe-Roberto-final-step}
		\Ent_\mu(f) 
		\le 
		\bigl(
		16C_\mu(1+\sqrt{\rho})^2
		+
		C_\rho
		\bigr)
		\calE_p(f).
	\end{equation}
	We conclude the result for arbitrary $f$ by the monotone convergence theorem (applied to the RHS of~\eqref{eq:Barthe-Roberto-final-step}) and Fatou's lemma (applied to the LHS of~\eqref{eq:Barthe-Roberto-final-step}).
\end{proof}

\subsection{Technical lemmas}
Recall the definition~\eqref{eq:H_p_def} of $H_p$ and note that for $p\in(0,1)$ and $x>0$, $H_p(x)=pp'(x-x^{1/p}-x^{1/p'}+1)$.
\begin{lemma}\label{L:H_p_properties}
	For any $p\in(0,1]$, the following properties are true:
	\begin{enumerate}[label=(\roman*)]
		\item \label{L:H_p_monotone} 
		$H_p$ is increasing and convex on $[1,\infty)$;
		\item \label{L:H_p_lower_bound}
		$H_p(x)\ge (\log x)^2$ for $x\ge 1$.
	\end{enumerate}
	If additionally $x\ge \lambda$ for some $\lambda>1$, then
	\begin{enumerate}[resume, label=(\roman*)]
		\item $H_p(x) \ge (x \log x)
		\cdot \min\bigl\{
				\frac{H_p(\lambda)}{\lambda\log\lambda},
				\frac{H'_p(\lambda)}{1+\log\lambda},
				1
			\bigr\}$;
		\label{L:H_p_xlogx}
		\item $H_p(x) \ge H_p(xc) \cdot 
		\min\bigl\{
			\frac{H_p(\lambda)}{H_p(\lambda c)},
			\frac{H_p'(\lambda)}{cH_p'(\lambda c)},
			c^{-1/p}
		\bigr\}$
		for any $c>1$;
		\label{L:H_p_cx}
		\item $
		\frac{H_p(y)}{y}
		\le
		\frac{\lambda}{\lambda-1}
		\frac{H_p(x)}{x}
		$
		for any $\lambda \le y \le x$.
		\label{L:H_p_c_increasing}
	\end{enumerate}
\end{lemma}
\begin{proof}
	We start with~\ref{L:H_p_monotone} and~\ref{L:H_p_lower_bound}.
	
	For $p=1$, $H_1(x)=(x-1)\log x$, $H_1(1) = H_1'(1)=0$ and $H_1''(x) = \frac{1}{x^2} + \frac{1}{x} > 0$, yielding~\ref{L:H_p_monotone}, while \ref{L:H_p_lower_bound} follows immediately as $\log x < x-1$.
	
	For $p\in(0,1)$, denote $h(x)=(\log x)^2$. 
	Then, $H_p(1)= H_p'(1) = h(1) = h'(1) = 0$ and, using AM-GM inequality,
	\begin{equation}\label{eq:H_p_2nd_der}
		H''_p(x) 
		= 
		\frac{x^{1/p} + x^{1/p'}}{x^2}
		\ge 
		\frac{2 x^{1/2p+1/2p'}}{x^2}
		\ge 
		\frac{2}{x^2}
		\ge 
		2\frac{1-\log x}{x^2}
		=h''(x),
	\end{equation}
	whence~\ref{L:H_p_lower_bound} follows.
	Moreover,~\eqref{eq:H_p_2nd_der} implies that $H_p''(x)\ge 0$, yielding also~\ref{L:H_p_monotone}.

	To see~\ref{L:H_p_xlogx}, denote the function on its RHS by $\tilde{h}(x)$ and note that by the definition of $\tilde{h}$, $H_p(\lambda) \ge \tilde{h}(\lambda)$ and $H_p'(\lambda) \ge \tilde{h}'(\lambda)$. 
	Since for any $x\ge 1$,
	\[
		H''_p(x)
		=
		\frac{x^{1/p} + x^{1/p'}}{x^2}
		\ge 
		x^{1/p - 2}
		\ge 
		\frac{1}{x}
		=
		(x\log x)'',
	\]
	then also $H''_p(x) \ge \tilde{h}''(x)$ for any $x\ge \lambda$ and~\ref{L:H_p_xlogx} follows (note that the calculation above also covers the case $p=1$, as then $1/p'=0$).

	Similarly, if $\tilde{h}(x)$ is the RHS of~\ref{L:H_p_cx}, then $H_p(\lambda) \ge \tilde{h}(\lambda)$ and $H_p'(\lambda) \ge \tilde{h}'(\lambda)$ by the definition of $\tilde{h}$. 
	As 
	\[
		H''_p(x)
		=
		\frac{x^{1/p} + x^{1/p'}}{x^2}
		\ge 
		c^{-1/p}
		\cdot 
		c^2
		\frac{(cx)^{1/p} + (cx)^{1/p'}}{(cx)^2}
		=
		\frac{d^2}{dx^2}
		\bigl(
			c^{-1/p}
			H_p(cx)
		\bigr),
	\]
	it follows that $H_p''(x)\ge \tilde{h}''(x)$ for any $x\ge \lambda$, yielding~\ref{L:H_p_cx}.

	Finally, since $H_p(1)=0$ and $H_p$ is convex by~\ref{L:H_p_monotone}, we have for any $\lambda \le y\le x$,
	\[
		\frac{H_p(y)}{y}
		\le
		\frac{H_p(y)}{y-1}
		\le 
		\frac{H_p(x)}{x-1}
		\le
		\frac{\lambda}{\lambda-1}
		\frac{H_p(x)}{x}
	\]
	yielding~\ref{L:H_p_c_increasing}.
\end{proof}
For $1<\rho\le x$ and $k\ge 1$, denote
\begin{equation}\label{eq:alpha_def}
	\alpha_{x,\rho}(k)
	=
	\inf\Bigl\{\,
		\sum_{s=0}^{k-1}
		H_p\bigl(
			\frac{g_{s+1}}{g_s}
		\bigr)
		\mu_s
		\; :\;
		g_0=1\le\ldots\le g_{k-1}
		< \rho \le x \le g_k
	\,\Bigr\}.
\end{equation}
This quantity plays a crucial role in providing sufficient condition for the $p$-log-Sobolev inequalities in Theorem~\ref{T:main}.
Its definition is partially inspired by an analogous quantity defined in~\cite{bartheroberto} in the continuous setting.

\begin{lemma}\label{L:alpha_estimate}
	For any $1<\rho\le x$ and $k\ge 1$,
	\begin{equation}\label{eq:alpha_1_estimate}
		\alpha_{x,\rho}(k)
		\ge 
		\Bigl[
			{\sum_{s=0}^{k-1} \mu_s^{-1}}	
		\Bigr]^{-1}
		\cdot
		{(\log x)^2}
	\end{equation}
	and 
	\begin{equation}\label{eq:alpha_2_estimate}
		\alpha_{x,\rho}(k)
		\ge 
		H_p(x\rho^{-1})\mu_{k-1}.
	\end{equation}
\end{lemma}
\begin{proof}
	We start with~\eqref{eq:alpha_1_estimate}.
	By Lemma~\ref{L:H_p_properties},~\ref{L:H_p_lower_bound},
	\begin{align*}
		\alpha_{x,\rho}(k)
		&\ge 
		\inf\Bigl\{
			\sum_{s=0}^{k-1}
			\Bigl(
				\log \frac{g_{s+1}}{g_s}
			\Bigr)^2 \mu_s
			\; :\;
			g_0=1\le\ldots\le g_{k-1}
			< \rho \le x \le g_k
		\Bigr\}
		\\&\ge 
		\inf\Bigl\{
			\sum_{s=0}^{k-1}
			\lambda_s^2 \mu_s
			\; :\;
			\sum_{s=0}^{k-1}\lambda_s \ge \log x
		\Bigr\}
		\\&\ge 
		\inf\Bigl\{
			\Bigl[\sum_{s=0}^{k-1}
			\lambda_s
			\Bigr]^2
			\Bigl[
				{\sum_{s=0}^{k-1} \mu_s^{-1}}	
			\Bigr]^{-1}
			\; :\;
			\sum_{s=0}^{k-1}\lambda_s \ge \log x
		\Bigr\}
		=
		\Bigl[
			{\sum_{s=0}^{k-1} \mu_s^{-1}}	
		\Bigr]^{-1}
		\cdot (\log x)^2,
	\end{align*}
	where in the last estimate we used the Cauchy-Schwarz inequality.

	We turn to~\eqref{eq:alpha_2_estimate}. 
	By Lemma~\ref{L:H_p_properties}, $H_p$ is non-negative and increasing on $[1,\infty)$, whence 
	\begin{align*}
		\alpha_{x.\rho}(k)
		&\ge 
		\inf\Bigl\{\,
			H_p\bigl(
				\frac{g_{k}}{g_{k-1}}
			\bigr)
			\mu_{k-1}
			\; :\;
			1\le g_{k-1}
			< \rho \le x \le g_k
		\,\Bigr\}
		=
		H_p(x\rho^{-1})\mu_{k-1}
	\end{align*}
	as desired.
\end{proof}
\section{Proof of Theorem~\ref{T:main}}
	\subsection{Sufficient condition}
	Fix $p\in(0,1]$ and assume that $\mu$ satisfies the Poincar\'{e} inequality~\eqref{eq:poincare_N} with constant $C_P<\infty$ and that $\hat{C}<\infty$, i.e., the condition~\eqref{eq:main_condition} holds. 
	Recall the definition~\eqref{eq:C_mu} of $C_\mu$ and that for fully supported measures, $C_\mu \le C_P/2\mu_0 < \infty$, cf. Theorem~\ref{T:Miclo} and~Remark~\ref{rk:hardy-poincare}.
	We show that $\mu$ satisfies the $p$-log-Sobolev inequality~\eqref{eq:pLS_birth_death} with constant $C$ bounded from above by a quantity depending on $C_\mu$ and $\hat{C}$ only.

	Define $\rho$ as 
	\begin{equation}\label{eq:rho_def}
		\rho = \min\Bigl(\bigl(
			\frac{1+C_{\mu}}{C_\mu}
		\bigr)^{1/3}, 
		2\Bigr).
	\end{equation}
	By Corollary~\ref{C:Barthe-Roberto}, it suffices to show that for any non-decreasing function $f\colon\NN\to\RR_+$ with a finite number of jumps,
	\begin{equation}\label{eq:increasing_fn_cond}
		\sum_{k\colon f(k)\ge \rho\EE f}
		f(k)\log\big(\frac{f(k)}{\EE_\mu f}\big) \mu_k
		\le 
		C_{\rho}\calE_p(f)
		=
		C_{\rho}
		\sum_{k=0}^\infty
		f(k) 
		H_p\bigl(
			\frac{f(k+1)}{f(k)}
		\bigr)
		\mu_k
	\end{equation}
	for some constant $C_{\rho}>0$ independent of $f$.
	By the homogeneity of~\eqref{eq:increasing_fn_cond}, we can assume that $f(0)=1$.
	Let us consider such $f$ and denote
	\begin{align*}
		\tau_{0}=0,
		\quad 
		\tau_k = \inf\{\, l > \tau_{k-1} \colon f(l) \ge \rho f(\tau_{k-1}) \,\}
	\end{align*}
	for $k\ge 1$.
	Since $f$ has a finite number of jumps, then there exists $M\in\NN\setminus\{0\}$ such that $\tau_{M-1} < \tau_{M} =\infty$.
	If $M=1$, then the LHS of~\eqref{eq:increasing_fn_cond} equals zero as $\EE f \ge f(0)=1$ and~\eqref{eq:increasing_fn_cond} holds with any $C_{\rho}>0$.
	Assume therefore from now on that $M > 1$.
	For $k\in \{1,\ldots, M-1\}$, denote 
	\begin{equation*}
		\gamma_k = 
		\rho f(\tau_k) 
		\log_+\big(\frac{\rho f(\tau_k)}{\EE_\mu f}\big) 
		\mu[\tau_k,\infty),
	\end{equation*}
	where $\log_+(x) := \max(\log x, 0)$
	and let
	\begin{equation*}
		\delta_k = f(\tau_{k-1})
		\sum_{l=\tau_{k-1}}^{\tau_k-1} 
		H_p\bigl(
			\frac{f(l+1)}{f(l)}
		\bigr)
		\mu_l.
	\end{equation*}
	Since $\EE f \ge f(0) = 1$ and as for $k\ge 1$ and $l\in [\tau_k,\tau_{k+1})$, $f(\tau_k)\le f(l) < \rho f(\tau_k)$, we have

	\begin{align*}
		\sum_{k\colon f(k)\ge \rho\EE_\mu f}
		f(k)\log\big(\frac{f(k)}{\EE_\mu f}\big) \mu_k
		&\le 
		\sum_{k\colon f(k)\ge \rho}
		f(k)\log_+\big(\frac{f(k)}{\EE_\mu f}\big) \mu_k
		\\&=
		\sum_{k=1}^{M-1} \sum_{l=\tau_k}^{\tau_{k+1}-1}
		f(l)\log_+\big(\frac{f(l)}{\EE_\mu f}\big) \mu_l
		\\&<
		\sum_{k=1}^{M-1}
		\rho f(\tau_k) \log_+\big(\frac{\rho f(\tau_k)}{\EE_\mu f}\big) \mu[\tau_k,\tau_{k+1})
		\le 
		\sum_{k=1}^{M-1}
		\gamma_k
	\end{align*}
	and as $H_p\ge 0$ and by the monotonicity of $f$,
	\begin{align*}
		\sum_{k=0}^\infty
		f(k) H_p\bigl(
			\frac{f(k+1)}{f(k)}
		\bigr)
		\mu_k
		\ge
		\sum_{k=1}^{M-1} \delta_k.
	\end{align*}
	Therefore, to prove~\eqref{eq:increasing_fn_cond}, it suffices to show that 
	\begin{equation}\label{eq:necessary-cond-reduction-gamma-delta}
		\sum_{k=1}^{M-1} \gamma_k \le C_{\rho}\sum_{k=1}^{M-1}\delta_k.
	\end{equation}

	Recall that $H_p(1)=0$, $H_p\ge 0$ and recall the definition~\eqref{eq:alpha_def} of $\alpha_{x,\rho}$.
	Consider $g_l=1$ for $l=0,1,\ldots,\tau_{k-1}$ and $g_l=f(l)/f(\tau_{k-1})$ for $l=\tau_{k-1}+1,\ldots,\tau_k$.
	Then, by the definition of $\delta_k$, we have for any $k\in \{1,\ldots, M-1\}$,
	\begin{align}\label{eq:alpha_delta_relation}
	\begin{split}
		\alpha_{f(\tau_k)/f(\tau_{k-1}), \rho}(\tau_k)
		\le 
		\delta_k / f(\tau_{k-1}).
	\end{split}
	\end{align}
	For $k=1$, using the monotonicity of $f$, estimate~\eqref{eq:alpha_delta_relation} and the fact that $f(\tau_0)=1 < \rho \le f(\tau_1)$, we get
	\begin{align}
	\label{eq:C1_def}
	\begin{split}
		\gamma_1
		&\le
		\rho f(\tau_1) \log_+\big(\frac{\rho f(\tau_1)}{\mu[0,\tau_1) + f(\tau_1)\mu[\tau_1,\infty)}\big) 
		\mu[\tau_1,\infty)
		\cdot 
		\frac{\delta_1}{\delta_1}
		\\&\le 
		\rho f(\tau_1)
		\log_+\big(\frac{\rho f(\tau_1)}{\mu[0,\tau_1) + f(\tau_1)\mu[\tau_1,\infty)}\big)
		\mu[\tau_1,\infty)
		\bigl[
			{\alpha_{f(\tau_1),\rho}(\tau_1)}
		\bigr]^{-1}
		\cdot 	
		\delta_1
		\\&\le 
		\sup 	\Bigl\{
			{\rho x\log_+\bigl(\frac{\rho x}{\mu[0,l)+x\mu[l,\infty)}\bigr) \mu[l,\infty)}
			\bigl[
				{\alpha_{x,\rho}(l)}
			\bigr]^{-1}
			\; : \; 
			x \ge \rho,\; l \ge 1
		\Bigr\}
		\cdot 	
		\delta_1
		\\&=:
		C_1\delta_1.
	\end{split}
	\end{align}
	
	For $2\le k \le M-1$ (if such exist), we estimate each $\gamma_k$ based on two cases.
	To that end, choose $\varepsilon = \rho^{-1}$ and note that by the definition~\eqref{eq:rho_def} of $\rho$
	and by Lemma~\ref{L:rho_lower_bound}, for any $l\ge 1$,
	\begin{equation}\label{eq:rho_eps_relation}
		\varepsilon
		<
		1
		<
		\rho
		<
		\rho^2 
		\le
		\varepsilon 
		\frac{
			\mu[l-1,\infty)
		}{
			\mu[l,\infty)
		}.
	\end{equation}
	Consider the case $f(\tau_k)/f(\tau_{k-1}) > \varepsilon \mu[\tau_{k}-1,\infty)/\mu[\tau_k,\infty)$. 
	Using monotonicity of $f$, Markov's inequality implies that $\frac{\rho f(\tau_k)}{\EE_\mu f}\le \frac{\rho}{\mu[\tau_k,\infty)}$.
	Using this estimate together with~\eqref{eq:alpha_delta_relation}, we get that
	\begin{align}
	\begin{split}
		\gamma_k
		&\le 
		{\rho f(\tau_k) \log \big(\frac{\rho }{\mu[\tau_k,\infty)}\big) \mu[\tau_k,\infty)}
		\cdot 
		\frac{\delta_k}{\delta_k}
		\\&\le 
		\frac{\rho f(\tau_k)}{f(\tau_{k-1})}
		{
			\log \bigl(
				\frac{\rho}{\mu[\tau_k,\infty)}
			\bigr)
			\mu[\tau_k,\infty)
		}
		\bigl[
			{\alpha}_{f(\tau_k)/f(\tau_{k-1}),\rho}(\tau_k)
		\bigr]^{-1}
		\cdot 
		\delta_k
		\\&\le 
		\sup\Bigl\{
			\frac{
				\rho x \log\bigl(
					\frac{\rho}{\mu[l,\infty)}
				\bigr)
				\mu[l,\infty)
			}{
				{\alpha}_{x,\rho}(l)
			}
			\,:\,
			l \ge 1,\,
			x > 
			\varepsilon 
			\frac{
				\mu[l-1,\infty)
			}{
				\mu[l,\infty)
			}
		\Bigr\}
		\cdot 
		\delta_k
		\\&=:
		C_2 \delta_k.
	\end{split}
	\label{eq:C2_def}
	\end{align}
	If $f(\tau_k)/f(\tau_{k-1}) \le \varepsilon \mu[\tau_{k}-1,\infty)/\mu[\tau_k,\infty)$, 
	then using the estimate $\log_+(xy)\le \log_+(x)+\log_+(y)$, we split
	\[
		\gamma_k
		\le
		\underbrace{
		\rho f(\tau_k) \log_+\big(\frac{\rho f(\tau_{k-1})}{\EE f}\big) \mu[\tau_k,\infty) 
		}_{=: A}
		+
		\underbrace{
		\rho f(\tau_k) \log_+\big(\frac{f(\tau_k)}{f(\tau_{k-1})}\big) \mu[\tau_k,\infty) 
		}_{=: B},
	\]
	and we estimate $A$ and $B$ separately using the assumption as follows:
	\begin{align}
	\begin{split}\label{eq:gamma_recursion}
		A 
		&\le 
		\varepsilon \rho 
		f(\tau_{k-1}) \log_+\big(\frac{\rho f(\tau_{k-1})}{\EE f}\big) \mu[\tau_{k}-1,\infty)
		\\&\le 
		\varepsilon \rho 
		f(\tau_{k-1}) \log_+\big(\frac{\rho f(\tau_{k-1})}{\EE f}\big) \mu[\tau_{k-1},\infty)
		=\varepsilon \gamma_{k-1}
	\end{split}
	\end{align}
	and, using monotonicity of $f$ and~\eqref{eq:alpha_delta_relation},
	\begin{align}
	\begin{split}
		B 
		&= 
		\rho f(\tau_k) \log\big(\frac{f(\tau_k) }{f(\tau_{k-1})}\big) \mu[\tau_k,\infty)
		\cdot
		\frac{\delta_k}{\delta_k}
		\\&\le 
		\frac{\rho f(\tau_k)}{f(\tau_{k-1})} 
		\log\big(\frac{f(\tau_k) }{f(\tau_{k-1})}\big) \mu[\tau_k,\infty)
		\bigl[
			{\alpha}_{f(\tau_k)/f(\tau_{k-1}),\rho}(\tau_k)
		\bigr]^{-1}
		\cdot 
		\delta_k
		\\&\le 
		\sup 
		\Bigl\{
			\frac{\rho x \log(x)}{\alpha_{x,\rho}(l)}\mu[l,\infty)
			\; : \; 
			l \ge 1,\, 
			\rho \le x \le \varepsilon\frac{\mu[{l-1},\infty)}{\mu[l,\infty)}
		\Bigr\}
		\cdot 
		\delta_k
		\\&=:
		C_3\delta_k.
	\end{split}
	\label{eq:C3_def}
	\end{align}
	Combine estimates~\eqref{eq:C2_def}, \eqref{eq:gamma_recursion} and \eqref{eq:C3_def} to get that for $k > 1$,
	\[
		\gamma_k
		\le 
		\varepsilon
		\gamma_{k-1}
		+
		(C_2+C_3)
		\delta_k.
	\]
	By~\eqref{eq:C1_def}, $\gamma_1\le C_1\delta_1$, whence
	\begin{equation}\label{eq:gamma-delta-relation}
		(1-\varepsilon)
		\sum_{k=1}^{M-1}\gamma_k
		\le 
		(C_1+C_2+C_3)
		\sum_{k=1}^{M-1}\delta_k.
	\end{equation}
	It suffices therefore to estimate the terms $C_1$, $C_2$ and $C_3$.
	\stoptocwriting
	\subsection*{Estimating \boldmath$C_3$}
	Recall the definition~\eqref{eq:C3_def} of $C_3$ and the relation~\eqref{eq:rho_eps_relation} -- we split the supremum based on whether $x\le\rho^2$ or not.
	By Lemma~\ref{L:alpha_estimate}, eq.~\eqref{eq:alpha_1_estimate}, and by the definition~\eqref{eq:C_mu} of $C_\mu$,
	\begin{multline}
		\label{eq:C3_1st_estimate}
		\sup
		\Bigl\{
			\frac{\rho x \log(x)}{\alpha_{x,\rho}(l)}\mu[l,\infty)
			\; : \; 
			l \ge 1,\, 
			\rho \le x \le \rho^2
		\Bigr\}
		\\ \le 
		\sup\Bigl\{
			\frac{\rho x }{\log x}
			\bigl[
				\sum_{s=0}^{l-1}\mu_s^{-1}
			\bigr]
			\mu[l,\infty)
			\; : \; 
			l \ge 1,\, 
			\rho \le x \le \rho^2
		\Bigr\}
		\le 
		C_\mu\frac{\rho^3}{\log \rho}. 
	\end{multline}
	Similarly, by Lemma~\ref{L:alpha_estimate}, eq.~\eqref{eq:alpha_2_estimate}, and by the definition~\eqref{eq:C_mu} of $C_\mu$,
	\begin{multline}
		\label{eq:C3_2nd_estimate}
		\sup
		\Bigl\{
			\frac{\rho x \log(x)}{\alpha_{x,\rho}(l)}\mu[l,\infty)
			\; : \; 
			l \ge 1,\, 
			\rho^2 \le x 
			\le 
			\varepsilon
			\frac{\mu[{l-1},\infty)}{\mu[l,\infty)}
		\Bigr\}
		\\ \le 
		\sup
		\Bigl\{
			\frac{\rho x \log(x)}{H_p(x\rho^{-1})}
			\frac{\mu[l,\infty)}{\mu_{l-1}}
			\; : \; 
			l \ge 1,\, 
			\rho^2 \le x 
			\le 
			\varepsilon
			\frac{\mu[{l-1},\infty)}{\mu[l,\infty)}
		\Bigr\}
		\le 
		C_\mu
		\sup_{x\ge \rho^2}
		\frac{\rho x \log(x)}{H_p(x\rho^{-1})}.
	\end{multline}
	Combining~\eqref{eq:C3_1st_estimate} and~\eqref{eq:C3_2nd_estimate} and using~\ref{L:H_p_xlogx} of Lemma~\ref{L:H_p_properties} with $\lambda=\rho$ and $x\rho^{-1} \ge \rho$ in place of $x$, we obtain 
	\begin{align}\label{eq:C3_bound}
	\begin{split}
		C_3 
		&\le
		C_\mu\Bigl(
			\frac{\rho^3}{\log \rho}
			+
			\max\bigl\{
				1,
				\frac{\rho\log\rho}{H_p(\rho)},
				\frac{1+\log \rho}{H_p'(\rho)}
			\bigr\}
			\cdot
			\sup_{x\ge \rho^2}
			\frac{\rho^2 \log x}{\log (x\rho^{-1})}
		\Bigr)
		\\&=
		C_\mu\Bigl(
			\frac{\rho^3}{\log \rho}
			+
			2
			\max\bigl\{
				1,
				\frac{\rho\log\rho}{H_p(\rho)},
				\frac{1+\log \rho}{H_p'(\rho)}
			\bigr\}
			\cdot
			\rho^2
		\Bigr).
	\end{split}
	\end{align}

	\subsection*{Estimating \boldmath$C_2$}
	Recall the definition~\eqref{eq:C2_def} of $C_2$ and the relation~\eqref{eq:rho_eps_relation}.
	By Lemma~\ref{L:alpha_estimate}, eq.~\eqref{eq:alpha_2_estimate} and using that $\rho\le 2$ by definition, we get
	\begin{align}\label{eq:C2_step1}
	\begin{split}
		C_2
		&\le 
		\sup\Bigl\{
			\frac{
				\rho x 
			}{
				H_p(x\rho^{-1})
			}
			\cdot 
			\log\bigl(
					\frac{\rho}{\mu[l,\infty)}
				\bigr)
			\cdot 
			\frac{
				\mu[l,\infty)
			}{
				\mu_{l-1}
			}
			\,:\,
			l \ge 1,\,
			x > 
			\varepsilon 
			\frac{
				\mu[l-1,\infty)
			}{
				\mu[l,\infty)
			}
		\Bigr\}
		\\&\le 
		\sup\Bigl\{
			\frac{
				\rho x 
			}{
				H_p(x\rho^{-1})
			}
			\cdot 
			\log\bigl(
					\frac{2}{\mu[l,\infty)}
				\bigr)
			\cdot 
			\frac{
				\mu[l,\infty)
			}{
				\mu_{l-1}
			}
			\,:\,
			l \ge 1,\,
			x > 
			\varepsilon 
			\frac{
				\mu[l-1,\infty)
			}{
				\mu[l,\infty)
			}
		\Bigr\}.
	\end{split}
	\end{align}
	For any $l\ge 1$ and $x >
	\varepsilon 
	\frac{
		\mu[l-1,\infty)
	}{
		\mu[l,\infty)
	}$, applying~\ref{L:H_p_c_increasing} of Lemma~\ref{L:H_p_properties} with $\rho^{-1}x\ge \rho$ in place of $x$, 
	$
		\frac{\varepsilon}{\rho}
		\frac{
			\mu[l-1,\infty)
		}{
			\mu[l,\infty)
		}
		\ge \rho
	$ in place of $y$ and $\lambda=\rho$ (recall~\eqref{eq:rho_eps_relation}) we get that 
	\[
		\frac{\rho}{\varepsilon}
		\frac{
			\mu[l,\infty)
		}{
			\mu[l-1,\infty)
		}
		H_p\bigl(
			\frac{\varepsilon}{\rho}
			\frac{
				\mu[l-1,\infty)
			}{
				\mu[l,\infty)
			}
		\bigr)
		\le 
		\frac{\rho^2}{\rho-1}
		\frac{
			H_p(x\rho^{-1})
		}{x},
	\]
	which after rearrangement (recall that $\varepsilon=\rho^{-1}$) is equivalent to 
	\[
		\frac{\rho x}{H_p(x\rho^{-1})}
		\cdot 
		{\mu[l,\infty)}
		\le 
		\frac{\rho}{\rho -1}
		\Big[ 
			H_p\bigl(
				\frac{1}{\rho^2}
				\frac{
					\mu[l-1,\infty)
				}{
					\mu[l,\infty)
				}
			\bigr)
		\Big]^{-1}
		\mu[l-1,\infty),
	\]
	which combined with Lemma~\ref{L:rho_lower_bound} allows estimating further~\eqref{eq:C2_step1} as follows:
	\begin{align}\label{eq:C2_step2}
	\begin{split}
		C_2 
		&\le 
		\frac{\rho}{\rho-1} 
		\sup_{l\ge 1} \Bigl\{
			\Bigl[
				H_p\bigl(
					\frac{1}{\rho^2}
					\frac{
						\mu[l-1,\infty)
					}{
						\mu[l,\infty)
					}
				\bigr)
			\Bigr]^{-1}
			\cdot 
			\log\bigl(
				\frac{2}{\mu[l,\infty)}
			\bigr)
			\cdot 
			\frac{\mu[l-1,\infty)}{\mu_{l-1}}
		\Bigr\}
		\\&\le 
		(1+C_\mu)
		\frac{\rho}{\rho-1}
		\sup_{l\ge 1} \Bigl\{
			\Bigl[
				H_p\bigl(
					\frac{1}{\rho^2}
					\frac{
						\mu[l-1,\infty)
					}{
						\mu[l,\infty)
					}
				\bigr)
			\Bigr]^{-1}
			\cdot 
			\log\bigl(
				\frac{2}{\mu[l,\infty)}
			\bigr)
		\Bigr\}.
	\end{split}
	\end{align}
	For any $l\ge 1$, applying~\ref{L:H_p_cx} of Lemma~\ref{L:H_p_properties} with $c=\rho^2$, $\lambda=\rho$ and 
	$
	\frac{1}{\rho^2}
	\frac{
		\mu[l-1,\infty)
	}{
		\mu[l,\infty)
	}\ge \rho$
	in place of $x$ 
	(recall~\eqref{eq:rho_eps_relation})
	gives 
	\[
		H_p\bigl(
			\frac{1}{\rho^2}
			\frac{
				\mu[l-1,\infty)
			}{
				\mu[l,\infty)
			}
		\bigr)
		\ge 
		H_p\bigl(
			\frac{
				\mu[l-1,\infty)
			}{
				\mu[l,\infty)
			}
		\bigr)
		\min\bigl\{
			\frac{H_p(\rho)}{H_p(\rho^3)},
			\frac{H_p'(\rho)}{\rho^2 H_p'(\rho^3)},
			\rho^{-2/p}
		\bigr\},
	\]
	which combined with~\eqref{eq:C2_step2} and assumption~\eqref{eq:main_condition} results in
	\begin{align}\label{eq:C2_bound}
		C_2
		&\le 
		(1+C_\mu)
		\frac{\rho}{\rho-1}
		\max\bigl\{
			\frac{H_p(\rho^3)}{H_p(\rho)},
			\frac{\rho^2H_p'(\rho^3)}{H_p'(\rho)},
			\rho^{2/p}	
		\bigr\}
		\cdot
		\hat{C}
		<\infty.
	\end{align}

	\subsection*{Estimating \boldmath$C_1$.} 
	Recall the definition~\eqref{eq:C1_def} of $C_1$.
	We use the same ideas as above by considering two cases.
	For any $l\ge 1$, if $\rho \le x\le \varepsilon \mu[l-1,\infty)/\mu[l,\infty)$, then
	\begin{align}\label{eq:C1-C3-estimate}
	\begin{split}
		\rho x\log\bigl(\frac{\rho x}{\mu[0,l)+x\mu[l,\infty)}\bigr) \mu[l,\infty)
		&\le 
		\rho x\log(\rho x) \mu[l,\infty)
		\\&\le 
		2\rho x\log( x) \mu[l,\infty)
		\le 
		2C_3 
		\cdot 
		\alpha_{x,\rho}(l),
	\end{split}
	\end{align}
	where in the fist and second step we used that $x\ge \rho > 1$ and in the last step we used the definition~\eqref{eq:C3_def} of $C_3$.

	If $x>\varepsilon \mu[l-1,\infty)/\mu[l,\infty)$, then
	by the definition~\eqref{eq:C2_def} of $C_2$,
	\begin{equation}\label{eq:C1-C2-estimate}
		\rho x\log\bigl(\frac{\rho x}{\mu[0,l)+x\mu[l,\infty)}\bigr) \mu[l,\infty)
		\le 
		\rho x\log\bigl(\frac{\rho }{\mu[l,\infty)}\bigr) \mu[l,\infty)
		\le 
		C_2
		\cdot 
		\alpha_{x,\rho}(l)
	\end{equation}
	and thus combining~\eqref{eq:C1-C3-estimate} and~\eqref{eq:C1-C2-estimate}, we arrive at
	\begin{equation}\label{eq:C1-C2-C3-estimate}
		C_1\le C_2+2C_3.
	\end{equation}
	\subsection*{Final estimate}
	Combining~\eqref{eq:C3_bound} and~\eqref{eq:C2_bound} together with bounds~\eqref{eq:gamma-delta-relation} and~\eqref{eq:C1-C2-C3-estimate} yields~\eqref{eq:necessary-cond-reduction-gamma-delta} with $C_\rho=\frac{2C_2+3C_3}{1-\varepsilon}$, which is bounded from above by
	\begin{multline*}
			\frac{\rho}{\rho-1}
			\Bigl[
			2
			(1+C_\mu)
			\hat{C}
			\frac{\rho}{\rho-1}
			\max\bigl\{
				\frac{H_p(\rho^3)}{H_p(\rho)},
				\frac{\rho^2H_p'(\rho^3)}{H_p'(\rho)},
				\rho^{2/p}	
			\bigr\}
			\\ +
			3
			C_\mu\Bigl(
				\frac{\rho^3}{\log\rho} 
				+
				2\rho^2
				\max\bigl\{
					1,
					\frac{\rho\log\rho}{H_p(\rho)},
					\frac{1+\log \rho}{H_p'(\rho)}
				\bigr\}
			\Bigr)
			\Bigr]
			< \infty.
	\end{multline*}
	Therefore, we obtain~\eqref{eq:increasing_fn_cond} for any non-decreasing $f\colon\NN\to\RR_+$ with a finite number of jumps and $C_\rho$ as above.
	Thus, Corollary~\ref{C:Barthe-Roberto} implies that
	\[
		\Ent_\mu(f)
		\le 
		[64C_\mu(1+\sqrt{\rho})^2 
		+ 
		4C_\rho]
		\calE_p(f)
	\]
	for any $f\colon\NN\to\RR_+$, as desired.
\resumetocwriting
\subsection{Necessary condition}
For the sake of contradiction, assume that there exists a sequence $0=\tau_0<\tau_1<\ldots$ such that~\eqref{eq:necessary_cond} holds, i.e.,  
\[
	\beta_k 
	:=  
	\Bigl[
		H_p\bigl( 
			\frac{\mu[\tau_{k-1},\infty)}{\mu[\tau_k,\infty)}
		\bigr)
		\Bigr]^{-1}
		\cdot 
		\frac{\mu[\tau_{k-1},\infty)}{\mu[\tau_{k}-1,\infty)}
		\cdot 
		\log\bigl(
			\frac{2}{\mu[\tau_k,\infty)}
		\bigr)
		\to
		\infty
\]
as $k\to\infty$,
and that $\mu$ verifies $p$-LS($C$) with some finite constant $C>0$.
By Theorem~\ref{T:MOS}, $\mu$ satisfies the Poincar\'{e} inequality~\eqref{eq:Poincare} with the same constant $C$
and therefore $C_\mu<\infty$ (recall the definition~\eqref{eq:C_mu} of $C_\mu$ and Remark~\ref{rk:hardy-poincare}).
For $M\ge 1$, set 
\[
	f_M = \Bigl[
		\sum_{k=0}^{M-1} \frac{
			{\bf 1}_{[\tau_k, \tau_{k+1})}
		}{
			\mu[\tau_k,\infty)
		}
	\Bigr]
	+
	\frac{
		{\bf 1}_{[\tau_M, \infty)}
	}{
		\mu[\tau_M,\infty)
	}.
\] 
Then $\EE_\mu f_M \le M+1$.
Moreover, by Lemma~\ref{L:rho_lower_bound}, $\frac{\mu[\tau_k,\tau_{k+1})}{\mu[\tau_k,\infty)}\ge\frac{1}{1+C_\mu}$ for any $k\in\NN$, whence
\[
	\Ent_\mu(f_M)
	\ge 
	-(M+1)\log(M+1)
	+
	\frac{1}{1+C_\mu}
	\sum_{k=1}^{M}\log\bigl(
		\frac{1}{\mu[\tau_k,\infty)}
	\bigr).
\]
Similarly (recall that $H_p(1)=0$),
\begin{align*}
	\calE_p(f_M)
	&=
	\sum_{k=1}^M \mu_{\tau_k-1}f(\tau_k-1)H_p\bigl(
		\frac{f(\tau_k)}{f(\tau_k-1)}
	\bigr)
	\\&\le
	\sum_{k=1}^M
	\frac{\mu[\tau_k-1,\infty)}{\mu[\tau_{k-1},\infty)}
	H_p\bigl(
		\frac{f(\tau_k)}{f(\tau_k-1)}
	\bigr)
	=
	\sum_{k=1}^M \beta_k^{-1}
	\log\bigl(
		\frac{2}{\mu[\tau_k,\infty)}
	\bigr)
\end{align*}
and consequently, since $\Ent_\mu(f_M)\le \frac{C}{4}\calE_p(f_M)$ by assumption, 
\begin{multline}\label{eq:beta-M-estimate}
	\sum_{k=1}^M
	\Bigl[
		\log\bigl(
			\frac{1}{\mu[\tau_k,\infty)}
		\bigr)
		\cdot 
		\bigl(
			1 - \frac{C(1+C_\mu)}{4\beta_k}
		\bigr)
	\Bigr]
	\le 
	(1+C_\mu)\Bigl[
		\frac{C}{4}M(\log 2)\sup_{k}\beta_k^{-1}
		+
		(M+1)\log(M+1)
	\Bigr]
\end{multline}
(recall that $\beta_k\to\infty$, whence $\sup_{k}\beta_k^{-1}<\infty$).
Let $k_0$ be such that $2\beta_k\ge C(1+C_\mu)$ for every $k\ge k_0$.
By Proposition~\ref{P:poisson_tails}, 
$\log\bigl(
	\frac{1}{\mu[\tau_k,\infty)}
\bigr)\ge cp\tau_k\log\tau_k$ for some constant $c>0$, whence by~\eqref{eq:beta-M-estimate}
\begin{align*}
	\sum_{k=k_0}^{M}k\log k
	&\le 
	\sum_{k=k_0}^{M}\tau_k\log \tau_k
	\\&\le 
	\frac{1}{cp}
	\sum_{k=k_0}^M
	\log\bigl(
		\frac{1}{\mu[\tau_k,\infty)}
	\bigr)
	\le
	\tilde{c}
	\cdot 
	(M+1)\log(M+1)
\end{align*}
for $M$ big enough and some constant $\tilde{c}>0$ independent of $M$.
We arrive at the desired contradiction by taking a limit as $M\to\infty$ and by noting that $\int_{k_0}^M x\log x\,dx \ge \hat{c}M^2\log M$ for $M$ big enough and some constant $\hat{c}>0$ independent of $M$. 
\section{Proof of Theorem~\ref{C:main}}
As a counterexample we will take the Conway--Maxwell--Poisson distribution with parameter $\nu>0$, defined as $\mu_\nu(k) = \frac{1}{Z_\nu}(k!)^{-\nu}$ for $k \in \NN$, where $Z_\nu = \sum_{k\ge 0} (k!)^{-\nu}$ is the normalizing constant, cf.~\cite{conway1962queuing}.

For any $n\in\NN$, using the estimate 
\begin{equation}\label{eq:factorial-estimate}
	b! \ge a!(b-a)!
\end{equation} 
valid for $b \ge a \ge 0$ we obtain
\begin{align}\label{eq:CMP_est}
\begin{split}
	\mu_\nu(\{n\})
	\le 
	\mu_\nu[n,\infty)
	&=
	\frac{1}{Z_\nu}
	\sum_{k\ge n}
	\frac{1}{(k!)^\nu}
	\\&\le
	\frac{1}{Z_\nu (n!)^\nu}
	\sum_{k \ge 0}
	\frac{1}{(k!)^\nu}
	=
	\frac{1}{(n!)^\nu}
	=
	Z_\nu \mu_\nu(\{n\}).
\end{split}
\end{align}
Recall the definition~\eqref{eq:C_mu} of $C_\mu$.
Using~\eqref{eq:CMP_est} and~\eqref{eq:factorial-estimate}, we obtain for $n\in\NN\setminus\{0\}$,
\begin{align*}
	\mu_\nu[n,\infty)
	\sum_{k=0}^{n-1}
	\frac{1}{\mu_\nu(\{k\})}
	\le 
	Z_\nu \mu_\nu(\{n\})
	\sum_{k=0}^{n-1}
	\frac{1}{\mu_\nu(\{k\})}
	= 
	Z_\nu
	\frac{\sum_{k=0}^{n-1}(k!)^{\nu}}{(n!)^{\nu}}
	\le 
	Z_\nu^2,
\end{align*}
whence $C_{\mu_\nu}<\infty$ and thus $\mu_\nu$ satisfies the Poincar\'{e} inequality~\eqref{eq:poincare_N} with constant $C_P=8Z_\nu^2$, cf. Theorem~\ref{T:Miclo} and Remark~\ref{rk:hardy-poincare}.

We first show that for any $\nu\in(0,1]$, $\mu_\nu$ verifies the $p$-log-Sobolev inequality~\eqref{eq:pLS_birth_death} for any $p<\nu$.
Fix some $0<p<\nu\le 1$.
By the first part of Theorem~\ref{T:main}, it suffices to show that 
\begin{equation}\label{eq:CMP_cond}
	\sup_{n\ge 1}
	\Bigl\{
	\Bigl[
	H_p\bigl( 
		\frac{\mu_\nu[n-1,\infty)}{\mu_\nu[n,\infty)}
	\bigr)
	\Bigr]^{-1}
	\cdot 
	\log\bigl(
		\frac{2}{\mu_\nu[n,\infty)}
	\bigr)
	\Bigr\} 
	<\infty	.
\end{equation}
By~\eqref{eq:CMP_est}, for any $n\ge 1$
\begin{align}\label{eq:CMP_tails}
\begin{split}
	\frac{n^\nu}{Z_\nu} 
	= 
	\frac{\mu_\nu(\{n-1\})}{\mu_\nu(\{n\})} \frac{1}{Z_\nu} 
	&\le
	\frac{\mu_\nu(\{n-1\})}{\mu_\nu[n,\infty)}
	\\&\le
	\frac{\mu_\nu[n-1,\infty)}{\mu_\nu[n,\infty)}
	\\&\le 
	Z_\nu \frac{\mu_\nu(\{n-1\})}{\mu_\nu[n,\infty)}
	\le 
	Z_\nu \frac{\mu_\nu(\{n-1\})}{\mu_\nu(\{n\})}
	= 
	Z_\nu n^\nu.
\end{split}
\end{align}
This, together with Lemma~\ref{L:rho_lower_bound} and the definition~\eqref{eq:H_p_def} of $H_p$ implies that 
\begin{equation}\label{eq:CMP_Hp1}
	\Bigl[
	H_p\bigl( 
		\frac{\mu_\nu[n-1,\infty)}{\mu_\nu[n,\infty)}
	\bigr)
	\Bigr]^{-1}
	\le 
	\Bigl[
	H_p\bigl( 
		\max\bigl\{
			\frac{1+C_{\mu_\nu}}{C_{\mu_\nu}},
			\frac{n^\nu}{Z_\nu} 
		\bigr\}
	\bigr)
	\Bigr]^{-1}
	\le 
	\frac{C'}{n^{\nu/p}}
\end{equation}
for any $n\ge 1$ and some big enough constant $C'>0$ (independent on $n$ but dependent on $\nu$ and $p$).
By~\eqref{eq:CMP_est} and Stirling's formula
\begin{equation}\label{eq:stirling}
	\lim_{n\to\infty}
	\frac{
		\log \bigl(\frac{2}{\mu_\nu[n,\infty)}\bigr)
		}{
			n\log n
		}
	=
	\nu.
\end{equation}
Combining~\eqref{eq:CMP_Hp1} with~\eqref{eq:stirling} and recalling that $p<\nu$, we obtain~\eqref{eq:CMP_cond}.

Finally, we show that for any $\nu\in(0,1)$, $\mu_\nu$ does not satisfy the $\nu$-log-Sobolev inequality.
By Theorem~\ref{T:main}, it suffices to show that there exists an increasing sequence $\tau_0<\tau_1<\ldots$ such that 
\begin{equation}\label{eq:CMP_cond2}
	\lim_{n\to\infty}
	\Big\{
		\Bigl[
		H_\nu\bigl( 
			\frac{\mu[\tau_{n-1},\infty)}{\mu[\tau_n,\infty)}
		\bigr)
		\Bigr]^{-1}
		\cdot 
		\frac{\mu[\tau_{n-1},\infty)}{\mu[\tau_{n}-1,\infty)}
		\cdot 
		\log\bigl(
			\frac{2}{\mu[\tau_n,\infty)}
		\bigr)
	\Big\}
	=
	\infty.
\end{equation}
We choose $\tau_n = n$ and~\eqref{eq:CMP_cond2} becomes 
\begin{equation}\label{eq:CMP_cond3}
	\lim_{n\to\infty}
	\Big\{
		\Bigl[
		H_\nu\bigl( 
			\frac{\mu[{n-1},\infty)}{\mu[n,\infty)}
		\bigr)
		\Bigr]^{-1}
		\cdot 
		\log\bigl(
			\frac{2}{\mu[n,\infty)}
		\bigr)
	\Big\}
	=
	\infty.
\end{equation}
Analogously as in~\eqref{eq:CMP_Hp1}, using~\eqref{eq:CMP_tails} and the definition~\eqref{eq:H_p_def} of $H_p$, we get that 
\begin{equation}\label{eq:CMP_Hp2}
	\Bigl[
	H_\nu\bigl( 
		\frac{\mu_\nu[n-1,\infty)}{\mu_\nu[n,\infty)}
	\bigr)
	\Bigr]^{-1}
	\ge 
	\Bigl[
	H_\nu\bigl( 
		Z_\nu n^\nu
	\bigr)
	\Bigr]^{-1}
	\ge 
	\frac{C''}{n}
\end{equation}
for any $n\ge 1$ and some small enough constant $C''>0$ (independent on $n$ but dependent on $\nu$).
We conclude~\eqref{eq:CMP_cond3} by combining~\eqref{eq:stirling} with~\eqref{eq:CMP_Hp2}.
\section*{Acknowledgements}
I would like to thank Micha{\l} Strzelecki for telling me about Problem~\ref{Pr:MOS}.
I am also very grateful to Rados{\l}aw Adamczak, Krzysztof Oleszkiewicz and Micha{\l} Strzelecki for their valuable remarks on the preliminary versions of this manuscript.
\bibliographystyle{amsplain}
\bibliography{p-logSobolev}
\end{document}